\documentclass{amsart}
\usepackage{amsfonts,amscd,amsthm,amsgen,amsmath,amssymb}
\usepackage[all]{xy}
\usepackage{epsfig,color}
\usepackage[vcentermath]{youngtab}
\newtheorem{theorem}{Theorem}[section]
\newtheorem{lemma}[theorem]{Lemma}
\newtheorem{corollary}[theorem]{Corollary}
\newtheorem{proposition}[theorem]{Proposition}

\theoremstyle{definition}

\newtheorem{example}[theorem]{Example}

\theoremstyle{remark}
\newtheorem{remark}[theorem]{Remark}

\numberwithin{equation}{section}

\begin{document}

\title{On the slice genus of quasipositive knots in indefinite 4-manifolds}
\author{David Baraglia}

\address{School of Mathematical Sciences, The University of Adelaide, Adelaide SA 5005, Australia}

\email{david.baraglia@adelaide.edu.au}


\date{\today}

\begin{abstract}
Let $X$ be a closed indefinite $4$-manifold with $b_+(X) = 3 \; ({\rm mod} \; 4)$ and with non-vanishing mod $2$ Seiberg--Witten invariants. We prove a new lower bound on the genus of a properly embedded surface in $X \setminus B^4$ representing a given homology class and with boundary a quasipositive knot $K \subset S^3$. In the null-homologous case our inequality implies that the minimal genus of such a surface is equal to the slice genus of $K$. If $X$ is symplectic then our lower bound differs from the minimal genus by at most $1$ for any homology class that can be represented by a symplectic surface. Along the way, we also prove an extension of the adjunction inequality for closed $4$-manifolds to classes of negative self-intersection without requiring $X$ to be of simple type.
\end{abstract}

\maketitle


\section{Introduction}

An important problem in the study of smooth $4$-manifolds is to determine the minimal genus of an embedded surface representing a given homology class. The relative version of this problem for $4$-manifolds with boundary has also received considerable attention \cite{os1,nou,chh,sat,mmsw,cn,hera,mmp,kms,kmt,imt}.

Consider a closed, connected, oriented smooth $4$-manifold $X$ and let $X_0$ be the $4$-manifold with boundary obtained by removing an open ball from $X$. Given an oriented knot $K \subset S^3 = \partial X_0$ and a relative homology class $a \in H_2(X_0 , \partial X_0 ; \mathbb{Z}) \cong H_2( X ; \mathbb{Z})$, we seek the minimal genus of a properly embedded oriented surface $\Sigma \subset X$ with $\partial \Sigma = K$ and $[\Sigma] = a$. We define the {\em slice genus of $K$ with respect to $(X,a)$} to be the minimum genus of such a surface and denote it by $g_4(K,X,a)$ or $g_4(K,X_0,a)$. Of particular interest is the case that $[\Sigma]$ is null-homologous. In this case we call the minimal genus of such a surface the {\em $H$-slice  genus of $K$ in $X$} and denote it by $g_H(K,X)$ or $g_H(K,X_0)$. In the case that $X = S^4$, $X_0 = B^4$ is the $4$-ball and $g_H(K,X) = g_4(K)$ is the usual slice genus of $K$.

To motivate our results we first recall the following relative genus bound for surfaces in definite $4$-manifolds:

\begin{theorem}[Ozsv\'ath--Szab\'o \cite{os1}]\label{thm:osin}
Let $X_0$ be a smooth, compact, oriented, negative-definite $4$-manifold with $b_1(X_0) = 0$ and $\partial X_0 = S^3$. For any smooth, properly embedded surface $\Sigma \subset X_0$ bounding a knot $K$, we have
\[
2g(\Sigma) \ge [\Sigma]^2 + \left| [\Sigma] \right| + 2\tau(\Sigma).
\]
\end{theorem}
Here $\left| [\Sigma] \right|$ is defined as $\sum_{i=1}^{n} | \langle [\Sigma] , e_1 \rangle |$, where $e_1, \dots , e_n$ is an orthonormal basis for $H^2( X_0 ; \mathbb{Z})$, (which exists by Donaldson's diagonalisation theorem) and $\tau(K)$ is the Ozsv\'ath--Szab\'o tau-invariant, defined in \cite{os1}. Taking $\Sigma$ to be null-homologous, we obtain a lower bound on the $H$-slice genus of $K$:
\begin{equation}\label{equ:htau}
g_H(K,X_0) \ge \tau(K).
\end{equation}
A surprising feature of this bound is that it does not depend on the $4$-manifold $X_0$.

Suppose now that $K$ is quasipositive. This means that $K$ is the braid closure of a braid that is a product of the standard generators $\sigma_1, \dots , \sigma_{n-1}$ and their conjugates \cite{rud}. It is shown in \cite{hed} that as a consequence of the inequality given in \cite{pla}, one has $\tau(K) = g_4(K)$ for any quasipositive knot. For instance, if $T_{p,q}$ denotes the $(p,q)$-torus knot for positive coprime integers $p,q$, then $\tau(T_{p,q}) = g_4(T_{p,q}) = (p-1)(q-1)/2$. Thus for quasipositive knots the inequality (\ref{equ:htau}) gives $g_H(K,X) \ge g_4(K)$. But we obviously have $g_H(K,X) \le g_4(K)$ and so we deduce that the $H$-slice genus of any quasipositive knot in any negative definite $4$-manifold with $b_1(X)=0$ is equal to the slice genus. 

In this paper we prove an analogue of Theorem \ref{thm:osin} for quasipositive knots and links in a large class of indefinite $4$-manifolds. Let $L \subset S^3 = \partial B$ be an oriented link. Define the {\em Murasugi characteristic} $\chi_4(L)$ of $L$ to be the maximum of the Euler characteristic of any smooth, properly embedded, oriented surface in $B$ bounding $L$ and having no closed components \cite{mur,bofo}. If $L$ is a knot, then $\chi_4(L) = 1-2g_4(L)$, where $g_4(L)$ is the slice genus. 

Given a spin$^c$-structure $\mathfrak{s}$ on a compact, oriented smooth $4$-manifold $X$ with $b_+(X) > 1$, we let $SW(X , \mathfrak{s}) \in \mathbb{Z}$ denote the Seiberg--Witten invariant of $(X , \mathfrak{s})$. We also let
\[
d(X, \mathfrak{s}) = \frac{ c_1(\mathfrak{s})^2 - \sigma(X) }{4} - 1 +b_1(X) - b_+(X)
\]
denote the expected dimension of the Seiberg--Witten moduli space.

\begin{theorem}\label{thm:qpg}
Let $X$ be a smooth, closed, oriented $4$-manifold with $b_1(X)=0$ and $b_+(X) = 3 \; ({\rm mod} \; 4)$. Suppose that there is a spin$^c$-structure $\mathfrak{s}$ with $SW(X , \mathfrak{s}) = 1 \; ({\rm mod} \; 2)$. Let $X_0$ be the $4$-manifold with boundary $S^3$ obtained by removing an open ball from $X$ and let $L \subset S^3$ be a quasipositive link. Then for any connected, smooth, oriented, properly embedded surface $\Sigma \subset X_0$ bounding $L$, we have
\[
-\chi(\Sigma) \ge [\Sigma]^2 + | \langle [\Sigma] , c_1(\mathfrak{s}) \rangle | - \chi_4(L).
\]
In particular, if $L = K$ is a quasipositive knot, then
\[
2g(\Sigma) \ge [\Sigma]^2 + |\langle [\Sigma] , c_1(\mathfrak{s}) \rangle | + 2g_4(K).
\]
\end{theorem}

Taking $\Sigma$ in Theorem \ref{thm:qpg} to be null-homologous and $L$ to be a knot, we obtain:
\begin{corollary}\label{cor:g4}
Let $X$ and $K$ be as in Theorem \ref{thm:qpg}. Then $g_H( K , X ) = g_4(K)$.
\end{corollary}

In the case that $a \in H_2(X ; \mathbb{Z})$ can be represented by a closed surface $S$ for which the adjunction inequality is sharp in the sense that $2g(S)-2 = a^2 + \langle a , c_1(\mathfrak{s}) \rangle$, we find that the lower bound given by Theorem \ref{thm:qpg} differs from the minimal genus by at most $1$:

\begin{corollary}\label{cor:adj2}
Let $X,\mathfrak{s}$ and $K$ be as in Theorem \ref{thm:qpg}. If there exists a closed embedded surface $S \subset X$ representing the homology class $a \in H_2(X ; \mathbb{Z})$ and satisfying $2g(S)-2 = a^2 + \langle a , c_1(\mathfrak{s}) \rangle $, then we have
\[
0 \le g_4(K,X,a)  -\frac{a^2 + \langle a , c_1(\mathfrak{s}) \rangle }{2} - g_4(K) \le 1
\]

\end{corollary}

\begin{example}
Consider the elliptic surface $X = E(2n)$. We have $b_+(X) = 4n-1$ \cite[Proposition 3.1.11]{gs}, so $b_+(X) = 3 \; ({\rm mod} \; 4)$. Suppose that $a \in H_2(X ; \mathbb{Z})$ is orthogonal to the canonical class and $a^2 \ge -2$. Then according to \cite[Theorem 1.1]{ham}, $a$ can be represented by a surface of genus $g$, where $2g-2 = a^2$. Thus Corollary \ref{cor:adj2} gives
\[
0 \le g_4(K,E(2n),a) - \frac{1}{2}a^2 - g_4(K) \le 1
\]
for any quasipositive knot.
\end{example}

When $X$ is symplectic and the class $a \in H_2(X ; \mathbb{Z})$ can be represented by a symplectic surface, we can apply the positive solution of the symplectic Thom conjecture \cite{os2} to Corollary \ref{cor:adj2}:

\begin{theorem}\label{thm:sympl}
Let $X$ be a smooth, closed symplectic $4$-manifold with $b_1(X)=0$ and $b_+(X) = 3 \; ({\rm mod} \; 4)$. If the homology class $a \in H_2(X ; \mathbb{Z})$ can be represented by a closed embedded symplectic surface, then for any quasipositive knot $K$, we have
\[
0 \le g_4(K,X,a) - \frac{ a^2 + \langle a, K_X \rangle }{2} - g_4(K) \le 1,
\]
where $K_X$ is the canonical class of $X$.
\end{theorem}
\begin{proof}
Let $\mathfrak{s}$ be the anti-canonical spin$^c$-structure, so $c_1(\mathfrak{s}) = K_X$. Recall that for a symplectic $4$-manifold with $b_+(X)>1$, one has $SW(X,\mathfrak{s}) = \pm 1$. Let $S \subset X$ be a closed symplectic surface representing $a$. Then by the adjunction formula, we have $2g(S)-2 = a^2 + \langle a , K_X \rangle$. Hence we may apply Corollary \ref{cor:adj2}.
\end{proof}

\begin{remark}
In addition to the smooth $H$-slice genus $g_H(K,X)$, one can also define the {\em topological $H$-slice genus} $g_H^{top}(K,X)$ of $K$ in $X$, defined as the minimal genus of a connected, properly embedded, topologically locally flat, null-homologous surface in $X_0 = X \setminus B^4$ bounding $K$. When $X = S^4$, $g_H^{top}(K,S^4) = g_4^{top}(K)$ is the topological $4$-genus of $K$. For a torus knot $K = T_{p,q}$ where $p,q$ are positive coprime integers, it is known that $g_4^{top}(K) < g_4(K)$, except in the cases $K = T_{2,n}, T_{3,4}$ and $T_{3,5}$ \cite{rud2,bfll}. In fact, for such knots one has $g_4^{top}(K) \le \tfrac{6}{7}g_4(K)$ \cite{bfll}. From this and Corollary \ref{cor:g4} we deduce the following: let $K$ be a torus knot other than $T_{2,n}$, $T_{3,4}$ or $T_{3,5}$ and let $X$ be a smooth, closed, oriented $4$-manifold with $b_1(X)=0$, $b_+(X) = 3 \; ({\rm mod} \; 4)$ and having a non-vanishing mod $2$ Seiberg--Witten invariants. Then
\[
g_H^{top}(K,X) \le g_4^{top}(K) \le \frac{6}{7}g_4(K) = \frac{6}{7}g_H(K,X).
\]
In particular, for such an $X$ the difference $g_H(K,X) - g_H^{top}(K,X)$ can be arbitrarily large.
\end{remark}

Taking $L$ to be the unkot, Theorem \ref{thm:qpg} also implies an extension of the adjunction inequality to classes of negative self-intersection without requiring $X$ to be of simple type:

\begin{theorem}\label{thm:adjneg}
Let $X$ be a smooth, closed, oriented $4$-manifold with $b_1(X)=0$ and $b_+(X) = 3 \; ({\rm mod} \; 4)$. Let $\mathfrak{s}$ be a spin$^c$-structure on $X$ with $SW(X , \mathfrak{s}) = 1 \; ({\rm mod} \; 2)$. Suppose $\Sigma$ is a compact, connected, smoothly embedded surface in $X$ of genus $g(\Sigma)$. Then
\[
2g(\Sigma) \ge [\Sigma]^2 + | \langle [\Sigma] , c_1(\mathfrak{s}) \rangle |.
\]
\end{theorem}

Note that this is slighly weaker than the usual adjunction inequality when $[\Sigma]^2 > 0$, which takes the form $2g(\Sigma)-2 \ge [\Sigma]^2 + | \langle [\Sigma] , c_1(\mathfrak{s}) \rangle |$. When $X$ is of simple type the adjunction inequality in the stronger form with $2g(\Sigma)-2 \ge [\Sigma]^2 + | \langle [\Sigma] , c_1(\mathfrak{s}) \rangle |$ was proven in \cite{os2}.

Lastly, applying the same type of argument used in the proof of Theorem \ref{thm:qpg} to a negative definite $4$-manifold, we obtain an extension of the Ozsv\'ath--Szab\'o inequality to quasipositive links.

\begin{theorem}\label{thm:oslink}
Let $X$ be a smooth, compact, oriented negative-definite $4$-manifold with $b_1(X) = 0$ and $\partial X = S^3$. Let $L \subset S^3$ be a quasipositive link. Then for any connected, smooth, properly embedded surface $\Sigma \subset X$ bounding $L$, we have
\[
-\chi(\Sigma) \ge [\Sigma]^2 + \left| [\Sigma] \right| - \chi_4(L).
\]
\end{theorem}

We note here that an extension of the Ozsv\'ath--Szab\'o $\tau$ invariant for links was defined in \cite{cav} and one has that $-\chi(\Sigma) \ge 2\tau(L)-n$ for a smooth, connected, properly embedded surface $\Sigma$ in $X$ bounding $L$, where $n$ is the number of components of the link \cite[Proposition 1.4]{cav}. This suggests that there should be a generalisation of the Ozsv\'ath--Szab\'o inequality to arbitrary links taking the form $-\chi(\Sigma) \ge [\Sigma]^2 + \left| [\Sigma] \right| + 2\tau(L)-n$.

\subsection{Remark on orientations}

Throughout the paper knots and links are oriented. If $\Sigma$ is a properly embedded surface bounding a knot or link $L$, then we require $\Sigma$ to be oriented and that the induced orientation on $\partial \Sigma$ agrees with the given orientation on $L$.

\subsection{Structure of the paper}

In Section \ref{sec:qp} we prove some results concerning surfaces in $\mathbb{CP}^2 \setminus B^4$ bounding the mirror $\overline{L}$ of a quasipositive link $L \subset S^3$, leading to Lemma \ref{lem:c0}. In Section \ref{sec:adj} we first prove an extension of the Seiberg--Witten adjunction inequality to the case of non-vanishing Bauer--Furuta invariant (Proposition \ref{prop:bfadj}). We then use this and Lemma \ref{lem:c0} to prove the main results of the paper.

\subsection{Acknowledgements}

We thank Hokuto Konno for comments on a draft of this paper.

\section{Quasipositive links}\label{sec:qp}

Let $B_n$ denote the braid group on $n$ strands and $\sigma_1, \dots , \sigma_{n-1}$ the standard generators. A link $L$ is said to be {\em quasipositive} if it can be realised as the braid closure of a braid $\beta$ which is a product of $\sigma_1, \dots , \sigma_{n-1}$ and their conjugates \cite{rud}. Every quasipositive link can be realised as the transverse intersection of a plane algebraic curve $\Gamma \subset \mathbb{C}^2$ with a $3$-sphere $S^3 = \partial B \subset \mathbb{C}^2$ bounding a ball $B \subset \mathbb{C}^2$ \cite{rud}. Conversely, any link $L \subset S^3$ constructed in this manner is quasipositive \cite{bo}.

Let $L \subset S^3$ be a quasipositive link. Hence we can realise $L$ as the transverse intersection $L = \partial B \cap \Gamma$ of a plane algebraic curve $\Gamma \subset \mathbb{C}^2$ with an open ball $B \subset \mathbb{C}^2$. Let $\overline{\Gamma} \subset \mathbb{CP}^2$ be the projective completion of $\Gamma$. Here we identify $\mathbb{C}^2$ as the complement $\mathbb{C}^2 = \mathbb{CP}^2 \setminus L_{\infty}$ of a projective line $L_\infty$ in $\mathbb{CP}^2$.

\begin{lemma}
There exists a non-singular sextic curve $\Sigma \subset \mathbb{CP}^2$ which is disjoint from $\overline{B} \cup Sing(\overline{\Gamma})$ and which meets $\overline{\Gamma}$ transversally.
\end{lemma}
\begin{proof}
Let $P = \mathbb{P}( H^0( \mathbb{CP}^2 , \mathcal{O}(1) ) )$ be the dual projective space of $\mathbb{CP}^2$, the space of lines in $\mathbb{CP}^2$. Let $U \subset P$ be the set of lines disjoint from $\overline{B}$. This is an open subset of $P$ and is non-empty, since the line at infinity is disjoint from $\overline{B}$. For each $x \in Sing(\overline{\Gamma})$, let $\ell_x \subset P$ be the set of lines passing through $x$. Then $\ell_x$ is a projective line in $P$. Now since $Sing(\overline{\Gamma})$ is finite, it follows that $V = P \setminus  \left( \cup_{x \in Sing(\overline{\Gamma})} \ell_x \right)$ is a dense open subset of $P$. Hence $W = U \cap V$ is non-empty. So there exists a non-empty open subset $W$ of lines in $\mathbb{CP}^2$ disjoint from $\overline{B}$ and $Sing(\overline{\Gamma})$. 

Now consider the linear system $R = \mathbb{P}( H^0( \mathbb{CP}^2 , \mathcal{O}(6)))$ of sextic curves in $\mathbb{CP}^2$. Let $U_1 \subset R$ denote the open subset of sextic curves disjoint from $\overline{B} \cup Sing(\overline{\Gamma})$. Then $U_1$ is non-empty. Indeed, let $N$ be a projective line in $\mathbb{CP}^2$ disjoint from $\overline{B}$ and $Sing(\overline{\Gamma})$. Then the sextic with divisor $6N$ belongs to $U_1$. 

Choose a point $p \in \mathbb{CP}^2$ not lying $\overline{\Gamma}$. The space of lines in $\mathbb{CP}^2$ through $p$ can be identified with $\mathbb{CP}^1$ and this determines a regular map $\rho : \overline{\Gamma} \to \mathbb{CP}^1$ which sends a point $q \in \overline{\Gamma}$ to the line joining $p$ and $q$. Let $T \subset \mathbb{CP}^1$ be the finite set $\rho( Sing(\overline{\Gamma}))$. Then $\rho : \overline{\Gamma} \setminus \rho^{-1}(T) \to \mathbb{CP}^1 \setminus T$ is a branched covering. Any point in $\mathbb{CP}^1 \setminus T$ which is not a branch point of $\rho$ corresponds to a line in $\mathbb{CP}^2$ such that each point of intersection with $\overline{\Gamma}$ has multiplicity $1$. Let $U_2 \subset R$ be the set of sextics whose intersection multiplicities with $\overline{\Gamma}$ all equal $1$. Then $U_2$ is a non-empty open subset of $R$ because we can take a sextic which is the union of six lines corresponding to six distinct non-branch points of $\rho$ in $\mathbb{CP}^1 \setminus T$. So $U_1 \cap U_2$ is a non-empty open subset of $R$. Bertini's theorem implies that there exists a non-singular sextic which belongs to $U_1 \cap U_2$.
\end{proof}

Let $S \subset \mathbb{CP}^2$ be a non-singular sextic curve which is disjoint from $\overline{B} \cup Sing(\overline{\Gamma})$ and which meets $\overline{\Gamma}$ transversally. Note that $S$ is connected, since any two distinct components of $S$ would intersect in a singular point. Suppose $\overline{\Gamma}$ has degree $d$. Then $\overline{\Gamma}$ and $S$ meet in exactly $6d$ points, by B\'ezout's theorem. Let $W = \mathbb{CP}^2 \setminus B$. Then $W$ is a $4$-manifold with boundary $\partial W = S^3$. Consider $\Gamma_0 = \overline{\Gamma} \setminus B$. Then $\Gamma_0$ is a properly embedded surface in $X$ meeting $\partial X$ in the link $\overline{L} = \partial W \cap \Gamma_0$ (we obtain the mirror $\overline{L}$ of $L$, because the orientation on $S^3 = \partial W$ is opposite to the orientation obtained by viewing $S^3$ as the boundary of $B$).

\begin{lemma}\label{lem:c0}
There exists a smoothly embedded, connected, oriented surface $C_0 \subset W$ having the following properties:
\begin{itemize}
\item[(1)]{$C_0$ meets $\partial W$ transversally in $\overline{L}$,}
\item[(2)]{$C_0$ meets $S$ transversally in $6d$ points,}
\item[(3)]{$\chi(C_0) = 3d-d^2 - \chi_4(L)$,}
\item[(4)]{The homology class of $C_0$ in $H_2( W , \partial W ; \mathbb{Z}) \cong \mathbb{Z}$ is $d[\ell]$, where $[\ell]$ is the class of a projective line,}
\item[(5)]{Each connected component of $C_0$ meets $\overline{L}$.}
\end{itemize}
\end{lemma}
\begin{proof}
First note that it is enough to find a smoothly embedded, connected oriented surface $C_0 \subset W$ satisfying (2)-(5) and such that $C_0$ meets $\partial W$ transversally in a link which is isotopic to $\overline{L}$, for then we can perform an isotopy on $C_0$ supported in a neighbourhood of $\partial W$ so as to satisfy (1) while still maintaining conditions (2)-(5).

Let $C \subset \mathbb{CP}^2$ be a non-singular algebraic curve of degree $d$, obtained by deforming $\overline{\Gamma}$ within the space of degree $d$ algebraic curves. In particular, $C$ is connected. By choosing $C$ to be sufficiently close to $\overline{\Gamma}$, we can assume that the intersection of $C$ with $S^3 = \partial B$ remains transverse and that $C \cap \partial B$ is isotopic to $L$. Similarly, we can assume that $C$ meets $S$ transversally. Now take $C_0 = C \setminus B$. Then $C_0$ meets $\partial W$ transversally in a link which is isotopic to $\overline{L}$. 

We claim that $C_0$ satisfies (5). For if not, $C_0$ contains a connected component $U$ which has no boundary. Then $U$ is a complex submanifold of $\mathbb{CP}^2$. Chow's theorem implies that $U$ is algebraic, so $U$ is an irreducible component of $C$. But $C$ is irreducible, which implies $C = U$. But this is impossible as $C$ has non-empty intersection with $\overline{B}$.

$C$ is an algebraic curve of degree $d$, so $C$ and $S$ meet in $6d$ points. Since $S$ is disjoint from $\overline{B}$, it follows that $C_0$ satisfies (2). Moreover, this implies (4), since $S$ represents $6[\ell]$ and all the intersections of $S$ and $C_0$ are positive.

As $C$ is an algebraic curve of degree $d$, we have $\chi(C) = 3d-d^2$ by the degree-genus formula and hence $\chi(C_0) = 3d-d^2 - \chi( C \cap \overline{B} )$. So to prove (3) it remains to show that $\chi( C \cap \overline{B} ) = \chi_4(L)$. In fact, this proven in \cite{rud3}.
\end{proof}

\section{Adjunction inequalities}\label{sec:adj}

We will need an extension of the adjunction inequality to the case of non-vanishing Bauer--Furuta invariants. Recall that to a compact oriented smooth manifold $M$ with $b_1(M) = 0$ and a spin$^c$-structure $\mathfrak{s}$, one may define an invariant called the {\em Bauer--Furuta invariant}, which takes values in a certain equivariant stable cohomotopy group \cite{bf}. In what follows, we will concern ourselves only with the corresponding non-equivariant Bauer--Furuta invariant
\[
BF(M , \mathfrak{s}) \in \pi^{st}_{d(M,\mathfrak{s})+1},
\]
where $\pi^{st}_k$ denotes the $k$-th stable homotopy group of spheres and 
\[
d(M,\mathfrak{s}) = \frac{ c_1(\mathfrak{s})^2 - \sigma(M) }{4} - 1 + b_+(M)
\]
is the expected dimension of the Seiberg--Witten moduli space. 

The following result is the adjunction inequality for Bauer--Furuta invariants. Special cases of this result have appeared in \cite{fkm,fkmm,fkm2,mmp}.

\begin{proposition}\label{prop:bfadj}
Let $M$ be a smooth, compact oriented $4$-manifold with $b_1(M)=0$ and $b_+(M)>1$. Let $\mathfrak{s}$ be a spin$^c$-structure on $M$ for $BF(M , \mathfrak{s}) \neq 0$. Suppose $\Sigma$ is a compact, connected, smoothly embedded surface in $M$ and that $[\Sigma]^2 \ge 0$. Then:
\begin{itemize}
\item[(1)]{If $g \ge 1$ then
\[
2g(\Sigma)-2 \ge [\Sigma]^2 + | \langle [\Sigma] , c_1(\mathfrak{s}) \rangle |
\]
}
\item[(2)]{If $g=0$, then $[\Sigma]$ is a torsion class.}
\end{itemize}
\end{proposition}
\begin{proof}
The proof is similar to the case where the Seiberg--Witten invariant of $(M , \mathfrak{s})$ is non-zero. The main point is that since $BF(M , \mathfrak{s}) \neq 0$, it follows that the Seiberg--Witten equations for $(M , \mathfrak{s})$ admits a solution for any metric and any perturbation.

Consider first the case that $g \ge 1$ and $[\Sigma]^2 = 0$. The standard neck stretching argument of Kronheimer--Mrowka \cite{km1} implies that there exists a translation invariant solutions to the Seiberg--Witten equations on $\Sigma \times \mathbb{R}$, where $\Sigma$ is given a constant scalar curvature metric and $\mathbb{R}$ the standard Euclidean metric. As in \cite{km1}, this implies that $2g(\Sigma)-2 \ge \langle [\Sigma] , c_1(\mathfrak{s}) \rangle$. Reversing orientation on $\Sigma$ if necessary, we obtain $2g(\Sigma)-2 \ge  | \langle [\Sigma] , c_1(\mathfrak{s}) \rangle |$, which proves the result in this case.

The case $g=0$ was proven in \cite[Theorem 4.5]{mmp} (note however that the condition $[\Sigma] \neq 0$ in the statement of \cite[Theorem 4.5]{mmp} should be replaced with the stronger condition that $[\Sigma]$ is non-torsion. This can be seen by examination of the proof of \cite[Lemma 5.1]{fs}, on which the proof of \cite[Theorem 4.5]{mmp} is based).

Lastly, suppose that $g \ge 1$ and $[\Sigma]^2 = n > 0$. By possibly reversing orientation on $\Sigma$, we can assume that $\langle [\Sigma] , c_1(\mathfrak{s}) \rangle \ge 0$. Let $M_n$ be manifold obtained by blowing up $M$ at $n$ points disjoint from $\Sigma$. So $M_n$ is diffeomorphic to $M \#^n \overline{\mathbb{CP}^2}$. Let $\mathfrak{s}_i$ be a spin$^c$-structure on the $i$-th copy of $\overline{\mathbb{CP}^2}$ such that $S_i = c_1(\mathfrak{s}_i)$ represents the $i$-th exceptional divisor and let $\widetilde{\Sigma}$ denote the connected sum of $\Sigma$ with the $2$-spheres representing $-S_1, \dots , -S_n$. Then $\widetilde{\Sigma}$ has the same genus as $\Sigma$ and  $[\widetilde{\Sigma}]^2 = 0$. Define a spin$^c$-structure $\widetilde{\mathfrak{s}}$ on $M_n$ by gluing together the spin$^c$-structures $\mathfrak{s}, \mathfrak{s}_1, \dots , \mathfrak{s}_n$. Then $c_1(\widetilde{\mathfrak{s}}) = c + S_1 +  \cdots + S_n$. Since $\overline{\mathbb{CP}^2}$ is negative definite, has vanishing first Betti number and $d( \overline{\mathbb{CP}^2} , \mathfrak{s}_i ) = -1$, it follows that $BF(\overline{\mathbb{CP}^2} , \mathfrak{s}_i) \in \pi^{st}_0$ is the identity map. The connected sum formula for Bauer--Furuta invariants \cite[Theorem 1.1]{b2}, implies that $BF( M_n , \widetilde{\mathfrak{s}}) \neq 0$. Therefore, we are in the self-intersection zero case, so as shown above we have 
\[
2g(\Sigma)-2 \ge | \langle [\widetilde{\Sigma}] , c_1(\widetilde{\mathfrak{s}}) \rangle| = | \langle [\Sigma] , c_1(\mathfrak{s}) \rangle + n | = |\langle [\Sigma] , c_1(\mathfrak{s}) \rangle| + [\Sigma]^2,
\]
where we used $\langle [\Sigma] , c_1(\mathfrak{s}) \rangle \ge 0$ and $[\Sigma]^2 = n$. This is the adjunction inequality.
\end{proof}

\begin{theorem}\label{thm:qpg1}
Let $X$ be a smooth, compact, oriented $4$-manifold with $b_1(X)=0$ and $b_+(X) = 3 \; ({\rm mod} \; 4)$. Suppose that there is a spin$^c$-structure $\mathfrak{s}$ with $SW(X , \mathfrak{s}) = 1 \; ({\rm mod} \; 2)$. Let $X_0$ be the $4$-manifold with boundary $S^3$ obtained by removing an open ball from $X$. Let $L \subset S^3$ be a quasipositive link. Then for any connected, oriented, smooth, properly embedded surface $\Sigma \subset X_0$ bounding $L$, we have
\[
-\chi(\Sigma) \ge [\Sigma]^2 + | \langle [\Sigma] , c_1(\mathfrak{s}) \rangle | - \chi_4(L).
\]
\end{theorem}
\begin{proof}
First note that by blowing up $X$ at points disjoint from $\Sigma$, we can assume $d(X , \mathfrak{s}) = 0$. Doing this does not alter the value of $[\Sigma]^2 + | \langle [\Sigma] , c_1(\mathfrak{s}) \rangle |$.

Let $Y = X \# \mathbb{CP}^2$ be the connected sum of $\mathbb{CP}^2$ and $X$. More precisely, we obtain $Y$ by identifying $W = \mathbb{CP}^2 \setminus B$ and $X_0$ along their boundary $S^3$. Here we orient $S^3$ so that it is an ingoing boundary of $\mathbb{CP}^2$ and an outgoing boundary of $X_0$. Since $S$ is disjoint from $\overline{B}$, we may regard it an an embedded surface in $Y$. Let $\pi : \widetilde{Y} \to Y$ be the double cover of $Y$ branched over $S$. To see that the branched cover exists, first consider the double cover $\pi_Z : Z \to \mathbb{CP}^2$ branched along $S$. In fact $Z$ is a $K3$ surface \cite[Corollary 7.3.25]{gs} ($Z$ is a compact complex surface with $b_1(Z)=0$ and using the adjunction formula, one finds that the canonical bundle of $Z$ is trivial). Then $\pi^{-1}(B)$ consits of two balls in $Z$. By removing these balls and gluing in two copies of $X_0$, we obtain $\widetilde{Y}$. Further, this shows that $\widetilde{Y}$ is diffeomorphic to the connected sum $Z \# X \# X$ of a $K3$ surface and two copies of $X$.

Recall that $\pi^{st}_1 \cong \mathbb{Z}_2$ and is generated by the Hopf map $\eta : S^3 \to S^2$. According to \cite[Proposition 4.4]{b2}, if $b_+(M) = 3 \; ({\rm mod} \; 4)$ and $SW(M , \mathfrak{s}) = 1 \; ({\rm mod} \; 2)$, then $BF(M , \mathfrak{s}) = \eta \in \pi^{st}_1$. In particular, this is the case for $(X,\mathfrak{s})$ and also for the $K3$ surface $Z$, equipped with the unique spin$^c$ structure $\mathfrak{s}_Z$ which comes from a spin-structure. Let $\mathfrak{s}_{\widetilde{Y}}$ denote the spin$^c$-structure on $\widetilde{Y} = Z \# X \# X$ obtained by gluing together the spin$^c$-structures $\mathfrak{s}_Z , \mathfrak{s}, \mathfrak{s}$. The connected sum formula for Bauer--Furuta invariants implies that $BF( \widetilde{Y} , \mathfrak{s}_{\widetilde{Y}}) = \eta^3 \in \pi^{st}_3$. Now since $\eta^3 \neq 0$, we have that the Bauer--Furuta invariant of $(\widetilde{Y} , \mathfrak{s}_{\widetilde{Y}})$ is non-zero. 

Let $C_0 \subset W$ be as in Lemma \ref{lem:c0} and let $\Sigma' = \Sigma \cup_L C_0$ be the surface in $Y$ obtained by attaching $\Sigma$ and $C_0$ along their boundaries. We have that $\Sigma'$ is connected, since $\Sigma$ is connected and every component of $C_0$ meets $L$. Let $\widetilde{\Sigma} = \pi^{-1}(\Sigma')$. Since $C_0$ meets $S$ transversally, it follows that $\widetilde{\Sigma}$ is a smooth, compact, embedded surface in $\widetilde{Y}$ and that the restriction of $\pi : \widetilde{\Sigma} \to \Sigma$ is a branched double cover with $6d$ branch points. We have that
\[
[\widetilde{\Sigma}]^2 = 2[\Sigma']^2 = 2[C_0]^2 + 2[\Sigma]^2 = 2d^2 + 2[\Sigma]^2.
\]
Now we observe that for any given quasipositive link $L$, we can choose the algebraic curve $\Gamma$ to have arbitrarily large degree. Indeed, we can replace $\Gamma$ by the union of $\Gamma$ with any number of lines which are disjoint from $\overline{B}$. Thus we can take $d$ large enough that $[\widetilde{\Sigma}]^2 = 2d^2 + 2[\Sigma]^2 > 0$. Having chosen such a $\Gamma$, we may apply Proposition \ref{prop:bfadj} to obtain
\begin{align*}
-\chi(\widetilde{\Sigma}) &\ge [\widetilde{\Sigma}]^2 + | \langle [\widetilde{\Sigma}] , c_1( \mathfrak{s}_{\widetilde{Y}} ) \rangle |  \\
& \ge 2[\Sigma]^2 + 2d^2 + 2 | \langle [\Sigma] , c_1(\mathfrak{s}) \rangle |.
\end{align*}

Furthermore, by Riemann--Hurwitz, we have 
\[
\chi(\widetilde{\Sigma}) = 2\chi(\Sigma') - 6d = 2\chi(\Sigma) + 2\chi(C_0) - 6d = 2\chi(\Sigma)-2d^2 - 2\chi_4(L),
\]
where we used that $\chi(C_0) = 3d-d^2 - \chi_4(L)$ from Lemma \ref{lem:c0}. Putting these together, we get
\[
-\chi(\Sigma) + d^2 + \chi_4(L) \ge [\Sigma]^2 + d^2 + | \langle [\Sigma] , c_1(\mathfrak{s}) \rangle |,
\]
which gives the result.
\end{proof}

\begin{remark}
The proof of Theorem \ref{thm:qpg1} can be thought of as a generalisation of the proof of \cite[Corollary 1.3]{km2}.
\end{remark}

\begin{corollary}
Let $X$ be a smooth, compact, oriented $4$-manifold with $b_1(X)=0$ and $b_+(X) = 3 \; ({\rm mod} \; 4)$. Suppose that there is a spin$^c$-structure $\mathfrak{s}$ with $d(X , \mathfrak{s}) = 0$ and $SW(X , \mathfrak{s}) = 1 \; ({\rm mod} \; 2)$. Let $X_0$ be the $4$-manifold with boundary $S^3$ obtained by removing an open ball from $X$. Let $K \subset S^3$ be a quasipositive knot. Then for any connected, oriented, smooth, properly embedded, homologically trivial surface $\Sigma \subset X_0$ bounding $K$, we have $g(\Sigma) \ge g_4(K)$. Thus $g_4(K)$ is the minimal genus of a homologically trivial surface in $X$ bounding $K$.
\end{corollary}

\begin{remark}
We make some comments on the proof of Theorem \ref{thm:qpg1}:
\begin{itemize}
\item{In the proof of Theorem \ref{thm:qpg1} it is essential that we pass to the branched cover of $Y = X \# \mathbb{CP}^2$. This is because $b_+(\mathbb{CP}^2) = 1 \; ({\rm mod} \; 4)$ and so the adjunction inequality Proposition \ref{prop:bfadj} can not be directly applied to $Y$.}
\item{To obtain a bound on $\chi(\Sigma)$, one may consider a simpler strategy of capping off $X_0$ with a $4$-ball $B^4$ and closing up $\Sigma$ with a surface in $B^4$ with Euler characteristic equal to $\chi_4(L)$. This will give
\[
-\chi(\Sigma) \ge [\Sigma]^2 + | \langle [\Sigma] , c_1(\mathfrak{s}) \rangle | + \chi_4(L)
\]
provided $[\Sigma]^2 \ge 0$ and $[\Sigma]$ is non-torsion. Notice that this differs from Theorem \ref{thm:qpg1} in that the right hand side has $+\chi_4(L)$ whereas in Theorem \ref{thm:qpg1} we have $-\chi_4(L)$. This inequality is usually a much weaker bound on $\chi(\Sigma)$ than Theorem \ref{thm:qpg1}. For example, if $L$ is a knot then $\chi_4(L) = 1-2g_4(L)$, so $-\chi_4(L) \ge \chi_4(L)$, except when $g_4(L) = 0$.}
\item{Our inequality may be compared with similar adjunction-type inequalities in \cite{mmp}. For each spin$^c$-structure for which the Ozsv\'ath--Szab\'o mixed invariant $\Phi_{X,\mathfrak{s}}$ is non-zero, we get
\[
2g(\Sigma)-2 \ge [\Sigma]^2 + \langle [\Sigma] , c_1(\mathfrak{s}) \rangle - 2\nu^+(\overline{K})
\]
provided either $[\Sigma]^2 \ge 2\nu^+(\overline{K})$ or $X$ has Ozsv\'ath--Szab\'o simple type \cite[Theorem 1.1]{mmp}. Here $\nu^+(K)$ is the concordance invariant constructed by Hom and Wu \cite{howu}. Note that $\nu^+$ is always non-negative, so this inequality is weaker than Theorem \ref{thm:qpg1}, unless $2-2\nu^+(\overline{K}) > 2g_4(K)$ which can only happen if $K$ is slice.}
\item{In the proof of Theorem \ref{thm:qpg1}, it is crucial that $L$ is quasipositive. This ensures that the number of intersection points of $\Sigma' = \Sigma \cup_L C_0$ with $S$ is exactly $6d = |\langle [\Sigma'] , [S] \rangle|$. For non-quasipositive knots, Lemma \ref{lem:c0} must generally fail since there are knots which do not satisfy Theorem \ref{thm:qpg1}. For example, the left handed trefoil $K = \overline{T_{2,3}}$ is $H$-slice in $K3$ \cite[Example 2.5]{mmp}, but $g_4(K)=1$. This implies that Lemma \ref{lem:c0} fails for $K$.}
\end{itemize}
\end{remark}

\begin{proof}[Proof of Corollary \ref{cor:adj2}:]
Suppose $a$ can be represented by a surface $S \subset X$ satisfying $2g(S)-2 = a^2 + \langle a , c_1(\mathfrak{s}) \rangle$. Now let $X_0 = X \setminus B$, where the ball $B$ is chosen with $\overline{B}$ disjoint from $S$. Let $K \subset S^3$ be a quasipositive knot. Choose a connected, smooth, properly embedded surface $\Sigma \subset X_0$ bounding $K$, supported in a collar neighbourhood of $\partial X_0$ and having genus equal to $g_4(K)$. Let $\Sigma' \subset X_0$ be the surface obtained by attaching a handle joining $\Sigma$ to $S$. Then $g(\Sigma') = g(S)+g_4(K)$ and $\Sigma'$ is a properly embedded surface representing the class $a$ and bounding $K$. Therefore
\[
g_4(K,X,a) \le g(\Sigma') = g(S) + g_4(K) = \frac{a^2  + \langle a  , c_1(\mathfrak{s}) \rangle}{2} + g_4(K)+1.
\]
The inequality $g_4(K,X,a) \ge \dfrac{a^2  + \langle a  , c_1(\mathfrak{s}) \rangle}{2} + g_4(K)$ is obtained by applying Theorem \ref{thm:qpg1}.
\end{proof}

\begin{proof}[Proof of Theorem \ref{thm:oslink}:]
Let $\mathfrak{s}$ be a spin$^c$-structure on $X$ such that $d(X , \mathfrak{s}) = -1$. Since $X$ is negative definite this is equivalent to $c_1(\mathfrak{s})^2 = -b_2(X)$. Consider the $4$-manifold $X' = K3 \# X$ with spin$^c$-structure $\mathfrak{s}' = \mathfrak{s}_0 \# \mathfrak{s}$, where $\mathfrak{s}_0$ is the unique spin$^c$-structure on $K3$ coming from the spin-structure. Then $b_+(X') = 3$ and $SW(X' , \mathfrak{s}') = SW(K3 , \mathfrak{s}_0) = 1$, by the blowup formula for Seiberg--Witten invariants. So we may apply Theorem \ref{thm:qpg1} to deduce the inequality
\[
-\chi(\Sigma) \ge [\Sigma]^2 + | \langle [\Sigma] , c_1(\mathfrak{s}) \rangle | - \chi_4(L),
\]
for any spin$^c$-structure $\mathfrak{s}$ on $X$ for which $c_1(\mathfrak{s})^2 = -b_2(X)$. By Donaldson's diagonalisation theorem, the intersection form on $X$ is diagonalisable. From this, it is easily seen that the maximum of $| \langle [\Sigma] , c_1(\mathfrak{s}) \rangle |$ over all such spin$^c$-structures on $X$ equals $\left| [\Sigma]\right|$. Thus we find that
\[
-\chi(\Sigma) \ge [\Sigma]^2 + \left| [\Sigma] \right| -\chi_4(L).
\]
\end{proof}



\bibliographystyle{amsplain}

\end{document}